\documentclass[12pt]{amsart}

\usepackage{graphicx}

\newtheorem{lemma}{Lemma}
\newtheorem{prop}{Proposition}
\newtheorem{thm}{Theorem}

\newcommand{\lk}{\mathrm{lk}}

\title[Linking and knotting in digraphs]{Intrinsic linking and knotting are arbitrarily complex in directed graphs}
\author{Thomas W.~Mattman}
\address{Department of Mathematics and Statistics,
California State University, Chico,
Chico, CA 95929-0525}
\email{TMattman@CSUChico.edu}
\author{Ramin Naimi}
\author{Benjamin Pagano}
\address{Department of Mathematics,
Occidental College, Los Angeles, CA 90041}
\email{rnaimi@oxy.edu}
\email{bpagano@oxy.edu}

\subjclass[2010]{Primary 05C10, Secondary 57M15, 57M25, 05C20, 05C35 }
\keywords{intrinsically knotted graph, intrinsically linked graph, directed graph, spatial graph}

\begin{document}

\begin{abstract}
Fleming and Foisy \cite{ff} recently proved the existence of a digraph whose every embedding contains a $4$-component link, and left open the possibility that a directed graph with an intrinsic $n$-component link might exist. We show that, indeed, this is the case. In fact, much 
as Flapan, Mellor, and Naimi~\cite{FlaMelNai} show for graphs, 
knotting and linking are arbitrarily complex in directed graphs. Specifically, we 
prove the analog for digraphs of the main theorem of their paper:
for any $n$ and $\alpha$, every embedding of a sufficiently large complete digraph in 
$\mathbb{R}^3$ contains an oriented link with components $Q_1, \ldots, Q_n$ such that, for every $i \neq j$, 
$|\lk(Q_i,Q_j)| \geq \alpha$ and $|a_2(Q_i)| \geq \alpha$, where  $a_2(Q_i)$ denotes the 
second coefficient of the Conway polynomial of $Q_i$.
\end{abstract}

\date \today

\maketitle

\section{Introduction}

Fleming and Foisy \cite{ff} recently proved the existence of a digraph whose every embedding contains a $4$-component link, and left open the possibility that a directed graph with an intrinsic $n$-component link might exist. We show that, indeed, this is the case. In fact, much 
as Flapan, Mellor, and Naimi~\cite{FlaMelNai} show for graphs, 
knotting and linking are arbitrarily complex in directed graphs. Specifically, we 
prove the analog for digraphs of the two main theorems of their paper.

Before stating the results, we introduce some notation.
For graph $G$, the \emph{symmetric digraph} $\overline{DG}$ is obtained by replacing each edge $v_iv_j$ with two directed edges, $v_iv_j$ and $v_jv_i$. 
For cycle $C$ in a digraph, let $p_1, \dots, p_\delta$ in $C$ be maximal consistently directed paths (no $p_i$
is a subpath of a longer consistently directed path in $C$). Then $\delta$ is the \emph{directionality} of $C$. Thus, a consistently directed cycle is $1$--directional.
Following \cite{FlaMelNai}, we define the \emph{linking pattern} of a link of $n$ components, $L_1, \dots, L_n$, as the weighted graph on vertices $v_1, \dots, v_n$
where $| \lk(L_i,L_j)|$ (if nonzero) is the weight of $v_iv_j$. When the linking number is zero, there is no edge.
The \emph{mod 2 linking pattern} instead carries the weights $\omega(L_i,L_j) = \lk(L_1,L_j) \bmod 2$.

\begin{thm}
Let $\lambda$, $\delta \in \mathbb{N}$ with $\delta$ even or $1$.  For every $n \in \mathbb{N}$, there is a digraph $\overline{DG}$ such that every embedding of 
$\overline{DG}$ in 
$\mathbb{R}^3$ contains a link whose weighted linking pattern is $K_n$ with every weight at least $\lambda$ and every component $\delta$-directional.
\end{thm}

\begin{thm}
For every $n, \alpha \in \mathbb{N}$, there is a complete digraph $\overline{DK_r}$ such that every embedding of $\overline{DK_r}$ in $\mathbb{R}^3$ contains a link with $1$-directional components $Q_1, \dots, Q_n$ such that for every $i\neq j$, $|\lk(Q_i,Q_j)| \geq \alpha$ and $|a_2(Q_i)| \geq \alpha$. 
\end{thm}

Not just the statements of our theorem are similar to \cite{FlaMelNai}, but the proofs as well. We prove Theorem~1 in the next section and 
Theorem~2 in Section 3.

\section{Intrinsic Linking}

In this section we prove Theorem~1. After a couple of introductory lemmas, we follow the same path as in~\cite{FlaMelNai}.
Throughout this paper, indices are cyclic;
e.g., given say $x_i$ with $1 \le i \le n$ (or $0 \le i \le n$),
$x_{n+1}$ is to be understood as $x_1$ ($x_0$); 
and, more generally, for $i> n$,
$x_{i}$ is to be understood as $x_{i-n}$ ($x_{i -n-1}$).

\begin{lemma}
\label{newmanylinks}
Every spatial digraph $\overline{DK_{6m}}$  
contains at least $m$ pairwise disjoint 2-component links,
all with odd linking numbers,
such that all their components are $2$-directional.
\end{lemma}

\begin{proof}
Take an undirected graph $K_6$ with vertices $v_1,\dots,v_6$ and form the digraph $DG$ by 
orienting the edges such that $v_iv_j$ is directed from $v_i$ to $v_j$ when $i < j$.
Note that $DG$ contains no $1$-directional cycles.

Additionally, notice that all $3$-cycles in $DG$ must be $2$-directional.
As argued in~\cite{cg,s}, any embedding of $K_6$ contains a pair of $3$-cycles with odd linking number.
Then, any embedding of $DG$ must contain a link with $2$-directional 
components. As $DG$ is a subgraph of $\overline{DK_6}$
and  $\overline{DK_{6m}}$ contains $m$ distinct copies of $\overline{DK_6}$, 
every embedding of $\overline{DK_{6m}}$ in $\mathbb{R}^3$ contains
$m$ pairwise disjoint links with odd linking number and $2$-directional components.
\end{proof}

\begin{lemma}
\label{lem-01mtx}
Let $M$ be an $m \times n$ matrix with entries in $\mathbb{Z}_2$ where every column of $M$ contains at least one $1$.
For every $n$ and $m$, there exists a vector $v \in \mbox{ row}(M)$  for which over $\frac{n}{2}$ of the entries are $1$'s.
\end{lemma}

\begin{proof}
Let $\|v\|$ denote the number of $1$'s in a vector $v$. We proceed by induction on $n$. If $n=1$, the statement is obvious.

Fix $m>0$ and assume that for all $ 1 \leq l < n$, the statement holds for every $m \times l$ matrix.
Let $m$ be an $m \times n$ matrix and suppose $v_0$ is a vector in the row space that maximizes $\|v_0\|$. 
For a contradiction, assume that $\|v_0 \|=k \leq \frac{n}{2}$.
Without loss of generality, the first $k$ entries in $v_0$ are $1$, and the rest are all $0$.
Divide $M$ into the two matrices $M_L$ and $M_R$ where the $m \times k$ matrix $M_L$ is the first $k$ columns of $M$ and 
the $m \times n-k$ $M_R$ is the remaining $n-k$ columns of $M$.
Similarly, split vectors $x \in \mbox{row}(M)$ into vectors $x_L$ and $x_R$ of lengths $k$ and $n-k$ respectively.
For every $x$, $\|x_L\| \geq \|x_R\|$, because otherwise $\|v_0 + x\| > \|v_0\|$, contradicting the maximality of $\|v_0\|$.
By induction, there is an $x$ such that $\|x_R\| > \frac{n-k}{2}$. This would imply $\|x_L\| > \frac{n-k}{2}$, and $\|x\| > n-k \geq k = \|v_0|$,
a contradiction.
\end{proof}

\begin{lemma}
\label{bigZ}
Suppose a spatial digraph $\overline{DK_p}$ 
contains links with components $J_1, \cdots , J_{2n}$ and  $X_1, \cdots X_{2n}$ 
such that $J_i$ is 2-directional and $\omega (J_i,X_i) = 1$ for every $i \leq n$.
Then $\overline{DK_p}$ contains a 1-directional cycle $Z$ in $\overline{DK_p}$ with vertices on $J_1 \cup \cdots \cup J_{2n}$
such that for some  $I \subseteq \{1, \cdots, 2n\}$ with $| I | \ge n/2$, 
$\omega (Z, X_i) = 1$ for all $i \in I$.
Furthermore, for every  $\delta \ge 1$, $Z$ can be chosen to be $2\delta$-directional
if $\overline{DK_p}$ contains at least
$2\delta-2$  vertices  disjoint from all $J_i$ and  $X_i$.
\end{lemma}

\begin{figure}[h]
\includegraphics[width = 13cm]{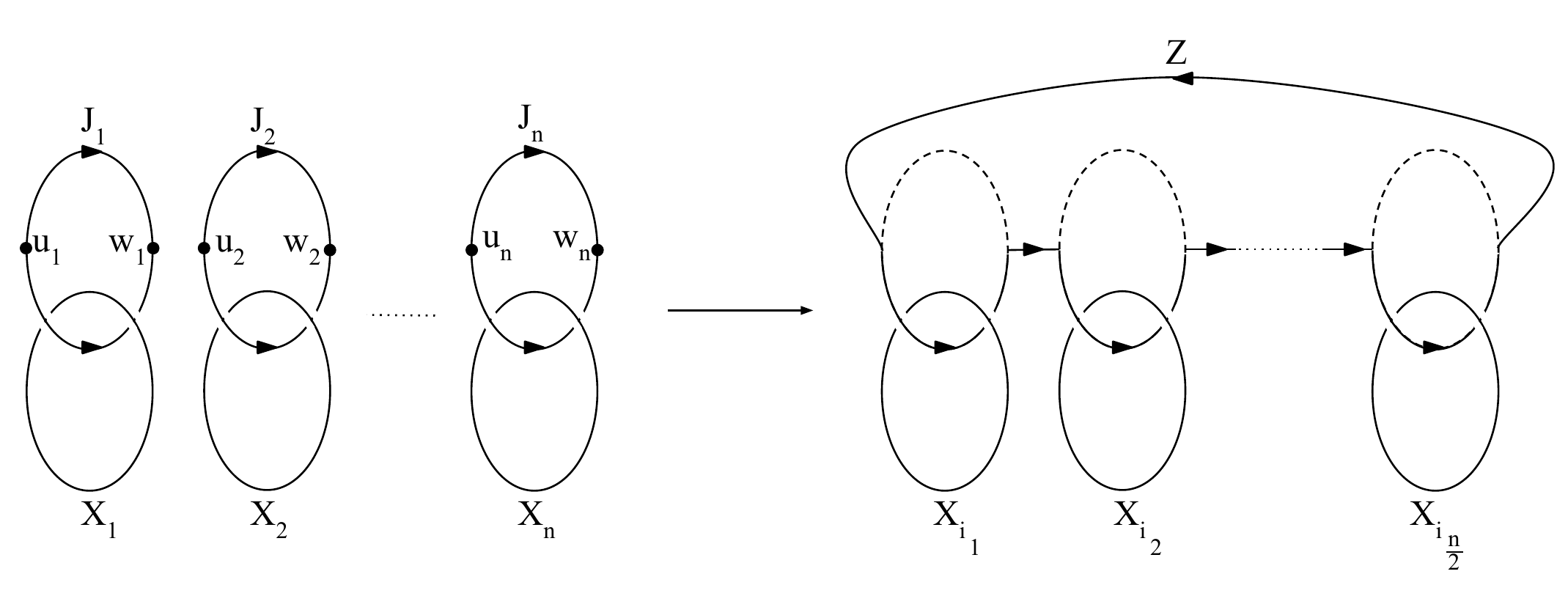}
\caption{Illustration of Lemma \ref{bigZ}}
\end{figure}

\begin{proof}
Since each $J_i$ is 2-directional,
it has exactly two vertices
where ``direction changes'' on $J_i$,
i.e., the two edges at each of these vertices 
are directed either both toward or both away from that vertex.
We label these two vertices $u_i$ and $w_i$,
so that both paths on $J_i$ between $u_i$ and $w_i$
are directed from $u_i$ toward $w_i$.
Let $q_i$ be one of these two directed paths;
let $w_i u_{i+1}$ denote the directed edge from $w_i$ to $u_{i+1}$;
and let $C = \bigcup q_i \cup \bigcup w_i u_{i+1}$.

Note that $C$ is a 1-directional cycle.
If we want the cycle $Z$ in the conclusion of the lemma to be $2$ directional,
we make $C$ $2$-directional by replacing the edge $w_{2n}u_1$ in $C$ with the edge $u_1w_{2n}$.
And if we want  $Z$ to be $2\delta$ directional with $\delta \ge 2$,
we replace the edge $w_{2n}u_1$ in $C$ with 
a path between $w_{2n}$ and $u_1$ 
that goes through the  $2\delta - 2$ extra vertices given in the hypothesis of the lemma,
such that the path changes direction at each of those vertices
and also at $w_{2n}$ and $u_1$.

If $\omega(C,X_i) = 1$ for at least $n/2$ of the $X_j$'s,
we let $Z = C$, and we are done.
Otherwise, we construct $Z$ as follows.
Let $M$ be the matrix with entries $M_{ij} = \omega(J_i,X_j)$.
Since $M_{ii} = 1$ for all $i$,
by Lemma~\ref{lem-01mtx},
there exist rows $i_1, \cdots ,i_k$ of $M$ whose sum contains greater than $n$ 1's.
Let  $Z = C\nabla  J_{i_1}\cdots \nabla J_{i_k}$.
Since $C$ links fewer than $n/2$ of the $X_j$'s,
while the sum of the rows $i_1, \cdots ,i_k$ contains greater than $n$ 1's,
it follows that $Z$ links at least $n/2$ of the $X_j$'s, as desired.
\end{proof}

Recall from \cite{FlaMelNai} that a \emph{generalized mod 2 keyring link} is one whose mod 2 linking pattern includes an $n$-star.

\begin{prop}
\label{prop1}
Let $n, \delta,$ and $ \epsilon \in \mathbb{N}$ with each of $\delta$ and $\epsilon$ either even or $1$. There is a digraph $\overline{DG}$ such that every embedding of $\overline{DG}$ in $\mathbb{R}^3$ contains a link whose mod 2 linking pattern contains the complete bipartite graph $K_{n,n}$, where every component of the first partition is $\delta$-directional and every component of the second partition is $\epsilon$-directional.
\end{prop}

\begin{proof}
The argument largely follows the proof of the corresponding proposition in \cite{FlaMelNai}, and 
we begin by summarizing their approach.
First observe that for any given $m$, there's a $p$ such that 
every embedding of $K_p$ contains $m$ disjoint mod 2 generalized keyrings, each having $n$ keys. 
Let $X_1, \dots, X_m$ denote the rings. Apply (the analogue of) Lemma~\ref{bigZ} $n$
times to construct $n$ cycles $Z_1, \dots, Z_n$ and index set $I_n$ of size at least $n$ so that,
for every $i \in I_n$ and every $j \leq n$, $\omega(Z_j, X_i) = 1$. The mod 2 linking pattern of the 
$Z_j$'s and $X_i$ with $i \in I_n$ then contains $K_{n,n}$.

It remains to modify the argument to take directionality into account. By combining 
Lemmas~\ref{newmanylinks} and \ref{bigZ}, we can find an embedding of $\overline{DK_{p}}$ with a desired
number $m$ of disjoint mod 2 generalized keyrings, such that each ring $X_i$ is $\epsilon$-directional. By applying Lemma~\ref{bigZ} $n$ times we again construct 
$I_n$ of cardinality at least $n$ and $\delta$-directional rings $Z_1, \dots, Z_n$ such that each $\omega(Z_j, X_i) = 1$. 
\end{proof}

\begin{lemma}
\label{bipar}
Let $\lambda \in \mathbb{N}$.
Let $\overline{DK_p}$ be embedded in $\mathbb{R}^3$ such that it contains a link 
with two-directional components $J_1, \cdots , J_r$, $L_1, \cdots , L_q$, $X_1, \cdots , X_m$, $Y_1, \cdots , Y_n$, 
where $r \geq m(2\lambda +1)2^m$, $q \geq (m+n)(2\lambda +1)3^m2^n$,  
and for every $i, j,\alpha, \beta$, $\mathrm{lk}(J_i,X_\alpha) \neq 0$ and $\mathrm{lk}(L_j, Y_\beta) \neq 0$.
Then $\overline{DK_p}$ 
contains a $1$-directional cycle $Z$ with vertices on $J_1\cup\cdots\cup J_r\cup L_1\cup\cdots\cup L_q$ 
such that for every $\alpha$ and $\beta$, $|\mathrm{lk}(Z, X_\alpha)| >\lambda$ and $|\mathrm{lk}(Z,Y_\beta)|>\lambda$.
Furthermore, for every  $\delta \ge 1$, $Z$ can be chosen to be $2\delta$-directional
if $\overline{DK_p}$ contains at leat
$2\delta-2$ vertices  disjoint from all $J_i, X_{\alpha}, L_j, Y_{\beta}$.
\end{lemma}
\begin{figure}[h]

\includegraphics[width = 13cm]{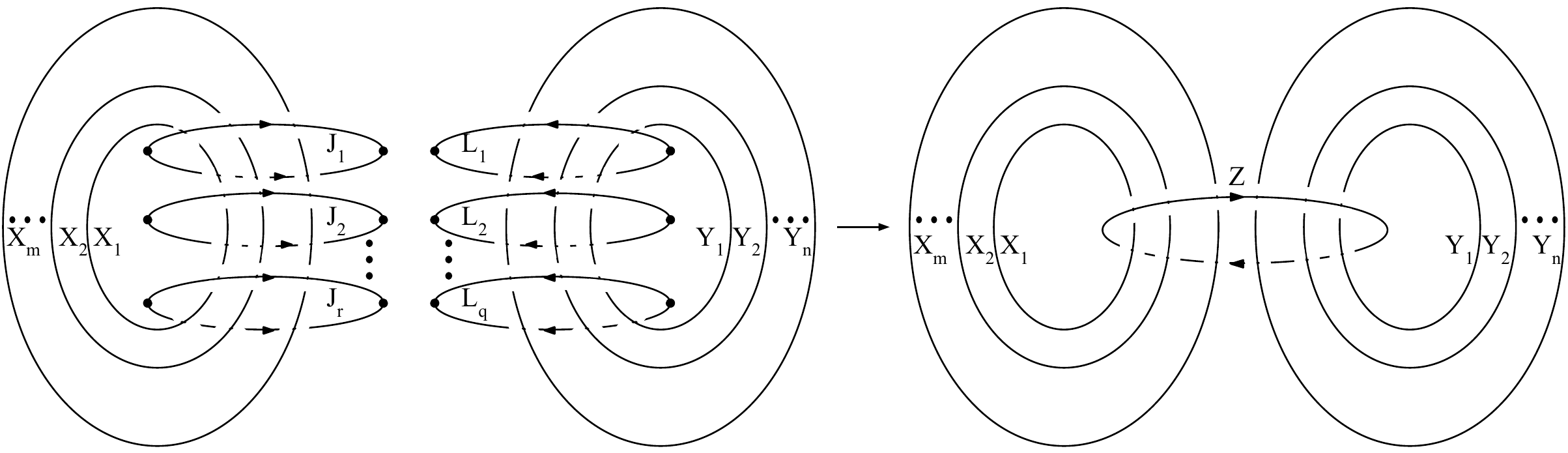}
\caption{Illustration of Lemma \ref{bipar}}

\end{figure}
\begin{proof}
The first step in the proof of the corresponding lemma in \cite{FlaMelNai} involves discarding $J_i$'s so that we are left with only positive linking numbers between each $J_i$ and each $X_\alpha$. To do this, note that at least half the linking numbers $\lk(J_i,X_1)$ have the same sign. If this sign is negative, we reverse the orientation of $X_1$ so that they become positive. Of the $J_i$'s which have positive linking number with $X_1$, at least half have the same signed linking number with $X_2$;  we repeat the process, eventually finding a set of at least $\frac{r}{2^m} \ge m(2\lambda + 1)$  $J_i$'s which each have positive linking number with every $X_\alpha$. 
We will assume without loss of generality that we are left with $J_1, \cdots, J_{m(2\lambda + 1)}$.

The same process can be used to find a set of $\frac{q}{2^n} \ge (m+n)(2\lambda + 1)3^m$ $L_j$'s which each have positive linking number with every $Y_\beta$.

For one final discard, we wish to throw out some of the remaining $L_j$'s so that for each $\alpha$, the linking numbers $\lk(L_j,X_\alpha)$ are either positive, negative, or zero for all $j$. 
For each $\alpha$, at least a third of the linking numbers $\lk(L_j,X_\alpha)$  falls into one of these three categories, 
so, after discarding the $L_j$'s in the two other categories, 
we  retain at least a third of the $L_j$'s. 
This process leaves us in the end with at least $\frac{(m+n)(2\lambda + 1)3^m}{3^m} = (m+n)(2\lambda + 1)$ $L_j$'s.
We will assume without loss of generality that we are left with $L_1, \cdots, L_{(m+n)(2\lambda + 1)}$.

Next, we create a cycle $C_0$ with vertices on $J_1,\cdots,J_{m(2\lambda +1)}$, and $L_1,\cdots,L_{(m+n)(2\lambda+1)}$. 
For $i \leq m(2\lambda +1)$, let $u_i$ and $w_i$ be the vertices on $J_i$ where direction changes, 
and for $j \leq (m+n)(2\lambda +1)$, let $u_{m(2\lambda+1)+j}$ and $w_{m(2\lambda+1)+j}$ be vertices on $L_j$ where direction changes. 
For $i \leq m(2\lambda+1)$, let $q_i$ be the path from $u_i$ to $w_i$ on $J_i$ which is directed opposite the orientation of $J_i$, 
and for $j \leq (m+n)(2\lambda+1)$, let $q_{m(2\lambda+1)+j}$ be the path from $u_{m(2\lambda +1)} + j$ to $w_{m(2\lambda +1)+j}$ on $L_j$ which is directed opposite the orientation of $L_j$. 
Additionally, for $k < (2m+n)(2\lambda +1)$, let $e_k$ be the edge in $\overline{DK_p}$ directed from $w_k$ to $u_k+1$.
For $k = (2m+n)(2\lambda +1)$,
let $e_{(2m+n)(2\lambda +1)}$ be the edge directed from $w_{(2m+n)(2\lambda+1)}$ to $u_1$  
if a one-directional cycle $Z$ is desired. 
If a $2$-directional cycle $Z$ is desired instead, 
let $e_{(2m+n)(2\lambda +1)}$ be the edge directed from $u_1$ to  $w_{(2m+n)(2\lambda+1)}$.  
And if a $2\delta$-directional cycle $Z$, where $\delta \ge 2$, is desired,
let $e_{(2m+n)(2\lambda+1)}$ be a path in $\overline{DK_p}$ from $w_{(2m+n)(2\lambda+1)}$ to $u_1$ 
that uses the $2\delta - 2$ additional vertices given in the hypothesis of the lemma,
such that direction changes at $w_{(2m+n)(2\lambda+1)}$, at $u_1$, and at every additional vertex.


Now, let $C_0$ be the union of all $e_k$ and $q_k$ for $1 \leq k \leq (2m+n)(2\lambda+1)$;
and let $C_s = C_0\nabla J_1\nabla\cdots\nabla J_s$, $1 \leq s \leq m(2\lambda +1)$.
We orient $C_0$  in the same direction as the $q_k$'s
 (i.e., on each arc $L_j \cap C_0$, $L_j$ and $C_0$ have opposite orientations),
 and the orientation of each $C_s$ is induced by that of $C_0$.
This implies  $\lk(C_{s+1},X_\alpha) > \lk(C_s,X_\alpha)$ for every $s$, including $s = 0$. 
Now, consider the matrix $A$ with entries $A_{\alpha, s} = \lk(C_s,X_\alpha)$,
where $1 \leq \alpha \leq m$, $0 \leq s \leq m(2\lambda +1)$.
Observe that the entries in each row of $A$ are pairwise distinct;
so in each row at most $2\lambda +1$ entries have magnitude less than or equal to $\lambda$.
Since $A$ has $m$ rows,  
it contains at most $m(2\lambda +1)$ entries which have magnitude less than $\lambda$.
On the other hand, $A$ has $m(2\lambda +1) +1$ columns,
which implies at least one of the columns has no entry less than or equal to $\lambda$.
In other words, for some $s$, $|\lk(C_s,X_\alpha)| > \lambda$ for every $\alpha$.
We let $D_0$ denote this cycle $C_s$.

Recall that for each $\alpha$, $\lk(X_\alpha, L_j)$ has the same sign ($+$, $-$, or 0) for every $j$.
For each $X_\alpha$, by reversing its orientation if necessary, 
we can assume its linking number with every $L_j$ is non-negative. 
Note that this does not change the fact that $|\lk(D_0,X_\alpha)| > \lambda$ for every $\alpha$.
Now, let $S$ be  the set of all $Y_\beta$'s and all $X_\alpha$'s which have positive linking number with every $L_j$. 
Thus, $S$ contains all $Y_\beta$'s and some $X_\alpha$'s --- 
at most $m+n$ cycles altogether.
For $1 \le t \leq (m+n)(2\lambda+1)$, let $D_t = D_0\nabla L_1\nabla\cdots\nabla L_t$ .
Then by a similar argument as above,
there is some $t$ such that $| \lk(D_t,A)| > \lambda$ for all $A \in S$. 
Let $Z$ denote this cycle $D_t$.
Observe that for each  $X_\alpha \not\in S$,
$|\lk(Z,X_\alpha)|  = |\lk(D_0,X_\alpha)| > \lambda$
since $\lk(X_\alpha, L_j) = 0$.
Thus $Z$ is our desired cycle.
\end{proof}

\begin{proof} (of Theorem 1)
The proof is largely similar to the corresponding theorem in \cite{FlaMelNai}. The main 
difference is we work with digraphs and must pay attention to the directionality of cycles.
For $m,n \in \mathbb{N}$, let $H(n,m)$ denote the complete $(n+2)$-partite graph with two parts ($P_1$ and $P_2$) of size $m$ and the remaining
parts ($Q_1, \dots, Q_n$) being single vertices. By induction on $n$, for every $n \geq 0$ and $m \geq 1$, 
we will show there is a digraph $\overline{DG}$ such that every embedding of $\overline{DG}$
includes a link whose linking pattern contains $H(n,m)$ with $Q_i$ to $Q_j$ edges of weight greater than $\lambda$ and each 
cycle represented by a $Q_i$ vertex $\delta$-directional. 

When $n = 0$, $H(0,m) = K_{m,m}$. By Proposition~\ref{prop1}, for every $m$, there is a digraph $\overline{DG}$ so that 
every embedding includes a link with linking pattern containing $K_{m,m}$. Moreover, we can assume that all cycles in the link 
are $2$-directional.

For the inductive step, assume that, 
for some $n \geq 0$ and every $m \geq 1$, there is a digraph $\overline{DG}$ such that every embedding if $\overline{DG}$ in $\mathbb{R}^3$ includes a link with linking pattern containing $H(n,m)$, where the weight of every $Q_i,Q_j$ edge exceeds $\lambda$, cycles for
vertices in $P_1$ and $P_2$ are $2$-directional, and those for $Q_i$ vertices are $\delta$-directional. Given $m$, 
let $q = (2m+n)(2\lambda + 1)3^m2^{m+n}$ and $s = m+q$. In the graph $H(n,s)$, label the $s$ vertices in $P_1$,  $X_1,\dots,X_m,L_1,\dots,L_q$ and those in $P_2$, $Y_1,\dots,Y_m,J_1,\dots,J_q$. For $i\leq n$ let $Y_{m+i}$ denote the vertex in$Q_i$.

By induction, there is a $\overline{DG}$ that includes, in each embedding, a link $L$ with linking pattern containing $H(n,s)$ with the desired weights and directionalities. Without loss of generality, we assume $\overline{DG}$ is a complete graph $\overline{DK_p}$. 
Fix an embedding of $\overline{DK_p}$ in $\mathbb{R}^3$.
We will show that this embedding also contains a link whose weighted linking pattern contains $H(n+1,m)$ with the desired weights, and whose components possess the desired directionality.

By abuse of notation, we denote the components of $L$ in $\overline{DK_p}$ by the name of the 
vertex that represents the component in the linking pattern.
Applying Lemma \ref{bipar} to the link in $\overline{DK_p}$ with components $J_1,\dots,J_q,L_1,\dots,L_q,X_1,\dots,X_m$, and $Y_1,\dots,Y_{m+n}$ where $r=q=(2m+n)(2\lambda+1)3^m2^{m+n}$, we find a $\delta$-directional cycle $Y_{m+n+1}$ where $|lk(Y_{m+n+1},X_\alpha)| > \lambda$ and $|lk(Y_{m+n+1},Y_\beta)| > \lambda$ for every $\alpha \leq n$ and $\beta \leq m+n$.

Thus, $\overline{DK_p}$ inludes a link $L'$ with components $X_1,\dots,X_m,$ and
$Y_1, \dots ,Y_{m+n+1}$, which can be partitioned into subsets corresponding to the vertices of $H(n+1,m)$. Namely $P'_1$ is the $X_i$ components,
$P'_2$ are the first $m$ $Y_i$'s and the remaining $Y_i$'s go, one each, to a $Q'_i$.
In $L'$, every component in one partition is linked with every component in all other partitions, each component in a partition $Q_i'$ is $\delta$-directional, and for every $i \neq j$ where $i,j \leq n+1$, $|lk(Y_{m+i},Y_{m+j})| > \lambda$. Thus, the weighted linking pattern of $L'$ contains $H(n+1,m)$, with the desired weights and directionality for every vertex and edge among partitions $Q_1',\dots,Q_{n+1}'$.

Therefore, we have proven that for every $n \geq 0$ and $m \geq 1$, there is a digraph $\overline{DG}$ whose every embedding in $\mathbb{R}^3$ contains a link whose linking pattern contains $H(n,m)$ where the weight of every edge between vertices $Y_{m+i}$ and $Y_{m+j}$ in $Q_i$ and $Q_j$ respectively is greater than $\lambda$, and the cycle represented by $Y_{m+i}$ is $\delta$-directional for every $i \leq n$. As $K_n$ is a subgraph of $H(n,m)$, we have shown that every embedding of $G$ in $\mathbb{R}^3$ contains a link whose linking pattern is $K_n$, the weight of every edge of $K_n$ being greater than $\lambda$, and every cycle represented by a vertex in $K_n$ being $\delta$-directional.
\end{proof}

\section{Intrinsic knotting}

In this section, we prove Theorem~2. Even more than what has gone before, we follow closely the argument of \cite{FlaMelNai}. 
We begin with two definitions from that paper.
{\em The weighted knotting and linking pattern} of the oriented link $L$ is the weighted linking
pattern along with the weight $|a_2(L_i)|$ on the vertex corresponding to component $L_i$, for each $i$. 
We use $J \nabla L$ for the closure of the 
symmetric difference and $J \nabla \epsilon L$ is $J \nabla L$ if $\epsilon = 1$ and $J \nabla \emptyset = J$, when $\epsilon = 0$.

The proof of the following three lemmas is virtually identical to those given in 
\cite{FlaMelNai} and we refer the reader there for details. The only
novelty is in the proof of the first lemma where, taking advantage of the $2$-directionality of $B_i$, 
we choose vertices $x_i$ and $y_j$ in $B_i$ so that both paths in $B_i$ are directed from $x_i$ to $y_j$.

\begin{lemma}
\label{b6}

Let $\lambda > 0$. In an embedding of $\overline{DK_r}$ in $\mathbb{R}^3$,  
let $A_1,\cdots,A_n$ be disjoint $1$-directional cycles and
 $B_1,\cdots,B_{6n+6}$ disjoint $2$-directional cycles such that $lk(A_h,B_i) \geq \lambda$ for all $h$ and $i$. Then there exist disjoint $2$-directional cycles $C_1,C_2,C_3,C_4 \in\{B_i\}$ and a $1$-directional cycle $W'$ in $\overline{DK_r}$ with verticies on $\bigcup_i B_i$ such that $W'$ intersects each $C_i$ in exactly one arc. In addition, $|lk(A_h,W'\nabla\epsilon_1C_1\nabla\epsilon_2C_2\nabla\epsilon_3C_3\nabla\epsilon_4C_4)| \geq \lambda$ for every $h$ and every choice of $\epsilon_1, \cdots, \epsilon_4 \in \{0,1\}$.
\end{lemma}

\begin{lemma}
\label{d4}

Let $\lambda > 0$.
In an embedding of $\overline{DK_r}$ in $\mathbb{R}^3$,  
 let $A_1,\cdots,A_n$ be disjoint $1$-directional cycles and let $B_1,\cdots,B_{6n+6}$ be disjoint $2$-directional cycles such that $lk(A_h,B_i) \geq \lambda$ and $|lk(B_i,B_j)| \geq \lambda$ for all $h, i$, and $j$. Then there exists a $1$-directional cycle $K$ in 
 $\overline{DK_r}$ with verticies on $\bigcup{i}B_i$ such that $|a_2(K)| \geq \lambda^2/16$ and $|lk(A_h,K)| \geq \lambda$ for every $h$.

\end{lemma}

\begin{lemma}
\label{final}

Let $n, \lambda \in \mathbb{N}$. Suppose that a complete graph $\overline{DK_r}$ embedded in $\mathbb{R}^3$ contains a link $L_0$ with $f^n(n)$ $2$-directional components, where $f(n) = n-1+(6n)2^{n-2}$, such that the linking number of every pair of components of $L_0$ has absolute value at least $\lambda$. Then $\overline{DK_r}$ contains a link with $1$-directional components $Q_1, \dots, Q_n$ such that for every $i\neq j, |lk(Q_i,Q_j)| \geq \lambda$ and $|a_2(Q_i)| \geq \lambda^2/16$.

\end{lemma}

\begin{proof}
(of Theorem 2)
As usual, the proof is very similar to that given in \cite{FlaMelNai}.
Let $\lambda = \mathrm{Max}\{\alpha,4\sqrt{\alpha}\}$, let $f(n) = n-1+(6n)2^{n-2}$, and let $m = f^n(n)$. By Theorem~1, there exists a graph $\overline{DK_r}$ such that every embedding of $\overline{DK_r}$ in $\mathbb{R}^3$ contains a link $L_0$ with $m$ $\delta$-directed components such that each pair of components has a linking number whose absolute value is at least $\lambda$. Let $\delta = 2$. By Lemma \ref{final}, every embedding of $\overline{DK_r}$ in $\mathbb{R}^3$ contains a link with $1$-directional components $Q_1,\dots,Q_n$  where, for every $i \neq j$, $|lk(Q_i,Q_j)| \geq \lambda \geq \alpha$ and $|a_2(Q_i)| \geq \lambda^2/16 \geq \alpha$.

\end{proof}

\end{document}